\documentclass[5p,times,twocolumn]{elsarticle}

\usepackage{lipsum}

\usepackage{lineno,hyperref}
\modulolinenumbers[5]

\usepackage{algorithm}  
\usepackage{algpseudocode}  
\usepackage{amsthm}
\usepackage{amsmath}
\usepackage{amsfonts}
\usepackage{bbm}
\usepackage{multirow}
\usepackage{chngpage}
\usepackage{verbatim} 
\usepackage{mathtools}
\usepackage{booktabs}
\usepackage{float}
\usepackage{array}
\newcolumntype{P}[1]{>{\centering\arraybackslash}p{#1}}
\usepackage{tabularx}
\usepackage{enumitem}
\usepackage{caption}
\usepackage{multicol}
\usepackage{xcolor}
\usepackage{soul}

\newtheorem{assumption}{Assumption}

\newtheorem{theorem}{Theorem}[section]
\newtheorem{corollary}{Corollary}[theorem]
\newtheorem{lemma}[theorem]{Lemma}
\newtheorem{remark}[theorem]{Remark}

\usepackage{hyperref}
\hypersetup{
    colorlinks=true,
    linkcolor=blue,
    filecolor=magenta,      
    urlcolor=cyan,
    citecolor={blue}
}

\journal{Journal of \LaTeX\ Templates}

\begin{document}

\begin{frontmatter}
\title{Asynchronous Parallel Stochastic Quasi-Newton Methods}

\author{Qianqian Tong$^1$}
\ead{qianqian.tong@uconn.edu}

\author{Guannan Liang$^1$}
\ead{guannan.liang@uconn.edu}

\author{Xingyu Cai$^2$}
\ead{xingyucai@baidu.com}

\author{Chunjiang Zhu$^1$}
\ead{chunjiang.zhu@uconn.edu}

\author{Jinbo Bi$^1$\corref{mycorrespondingauthor} }
\ead{jinbo.bi@uconn.edu}
\address{$^1$University of Connecticut, Storrs, CT 06269, $^2$Baidu USA, Sunnyvale, CA 94089}

\begin{abstract} 
  Although first-order stochastic algorithms, such as stochastic gradient descent, have been the main force to scale up machine learning models, such as deep neural nets, the second-order quasi-Newton methods start to draw attention due to their effectiveness in dealing with ill-conditioned optimization problems. The L-BFGS method is one of the most widely used quasi-Newton methods. We propose an asynchronous parallel algorithm for stochastic quasi-Newton (AsySQN) method. Unlike prior attempts, which parallelize only the calculation for gradient or the two-loop recursion of L-BFGS, our algorithm is the first one that truly parallelizes L-BFGS with a convergence guarantee. Adopting the variance reduction technique, a prior stochastic L-BFGS, which has not been designed for parallel computing, reaches a linear convergence rate. We prove that our asynchronous parallel scheme maintains the same linear convergence rate but achieves significant speedup. Empirical evaluations in both simulations and benchmark datasets demonstrate the speedup in comparison with the non-parallel stochastic L-BFGS, as well as the better performance than first-order methods in solving ill-conditioned problems.
\end{abstract}


\end{frontmatter}

\section{Introduction}
With the immense growth of data in modern life, developing parallel or distributed optimization algorithms has become a well-established strategy in machine learning, such as the widely used stochastic gradient descent (SGD) algorithm and its variants 
\cite{zinkevich2010parallelized,recht2011hogwild,duchi2011adaptive,defazio2014saga,kingma2014adam,reddi2015variance,schmidt2017minimizing}. Because gradient-based algorithms have deficiencies (e.g., zigzagging) with ill-conditioned optimization problems (elongated curvature), second-order algorithms that utilize curvature information have drawn research attention. The most well-known second-order algorithm is a set of quasi-Newton methods, particularly the Broyden-Fletcher-Goldfarb-Shanno (BFGS) method and its limited memory version (L-BFGS) \cite{dennis1977quasi,nocedal1980updating,dembo1982inexact,liu1989limited,bordes2009sgd,wang2017stochastic,mokhtari2018iqn,bottou2018optimization,karimireddy2018global,marteau2019globally,gao2019quasi,kovalev2020fast,jin2020non}. Second-order methods have several advantages over first-order methods, such as fast rate of local convergence (typically superlinear) \cite{dennis1974characterization}, and affine invariance (not sensitive to the choice of coordinates).

Stochastic algorithms have been extensively studied and substantially improved the scalability of machine learning models. Particularly, stochastic first-order methods have been a big success for which convergence is guaranteed when assuming the expectation of stochastic gradient is the true gradient. In contrast, second-order algorithms cannot directly use stochastic sampling techniques without losing the curvature information. Several schemes have been proposed to develop stochastic versions of quasi-Newton methods. Stochastic quasi-Newton method (SQN) \cite{byrd2016stochastic} uses independent large batches for updating Hessian inverse. Stochastic block BFGS \cite{gower2016stochastic} calculates subsampled Hessian-matrix product instead of Hessian-vector product to preserve more curvature information. These methods achieve a sub-linear convergence rate in the strongly convex setting. Using the variance reduction (VR) technique proposed in \cite{johnson2013accelerating}, the convergence rate can be lifted up to linear in the latest attempts \cite{moritz2016linearly,gower2016stochastic}. Later, acceleration strategies \cite{zhao2017stochastic} are combined with VR, non-uniform mini-batch subsampling, momentum calculation to derive a fast and practical stochastic algorithm. Another line of stochastic quasi-Newton studies tries to focus on solving self-concordant  functions, which requires more on the shape or property of objective functions, can reach a linear convergence rate.
Stochastic adaptive quasi-Newton methods for self-concordant functions have been proposed in \cite{zhou2017stochastic}, where the step size can be computed analytically, using only local information, and adapts itself to the local curvature.
\cite{meng2020fast} gives a global convergence rate for a stochastic variant of the BFGS method combined with stochastic gradient descent.
Other researchers use randomized BFGS \cite{kovalev2020fast}  as a variant of stochastic block BFGS and prove the linear convergence rate under self-concordant functions.
However, so far all these stochastic methods have not been designed for parallel computing.

\begin{table*}
\caption{The comparison of various quasi-Newton methods in terms of their stochastic, parallel, asynchronous frameworks, convergence rates, and if they use a variance reduction technique and limited memory update of BFGS method to make it work well for high dimensional problem. Here, our algorithm AsySQN is designed for paralleling the whole  L-BFGS method in shared memory. }\label{table1}
	\begin{center}
		\begin{tabular}{|c|c|c|c|c|c|c|}
			\hline 
			QN methods & Stochastic & Parallel &  Asy & VR &High dimensional & Convergence\\ 
			\hline\hline
			\cite{byrd2016stochastic} & $\checkmark$&  &  & & $\checkmark$& sublinear \\ 
			\cite{moritz2016linearly} & $\checkmark$ & & &$\checkmark$ &$\checkmark$& linear\\ 
			\cite{gower2016stochastic}  & $\checkmark$ & & &$\checkmark$ & $\checkmark$& linear  \\
			\cite{zhao2017stochastic} &$\checkmark$ & & &$\checkmark$  & $\checkmark$ & linear \\
			\cite{chen2014large}  & & parallel two-loop recursion &  & & $\checkmark$ & $-$ \\
			\cite{berahas2016multi} & $\checkmark$ & map reduce for gradient & & &$\checkmark$& sublinear \\
			\cite{bollapragada2018progressive} & $\checkmark$ & map reduce for gradient  & &$\checkmark$ & $\checkmark$ &linear  \\
			\cite{eisen2016decentralized,eisen2017decentralized} & & parallel calculation for gradient & $\checkmark$ & &  & $-$\\
			\cite{soori2020dave} & & parallel calculation for Hessian & $\checkmark$ & & &superlinear\\
			AsySQN &$\checkmark$ & parallel model for L-BFGS &$\checkmark$ &$\checkmark$ &$\checkmark$ &linear  \\
			\hline 
		\end{tabular}
	\end{center}
\end{table*}

Parallel quasi-Newton methods have been explored in several directions: map-reduce (vector-free L-BFGS) \cite{chen2014large} has been used to parallelize the two-loop recursion (See more discussion in Section 2) in a deterministic way; the distributed L-BFGS \cite{najafabadi2017large} is focused on the implementation of L-BFGS over high performance computing cluster (HPCC) platform, e.g. how to distribute data such that a full gradient or the two-loop recursion can be calculated fast. 
As a successful trial to create both stochastic and parallel algorithms, multi-batch L-BFGS \cite{berahas2016multi} uses map-reduce to compute both gradients and updating rules for L-BFGS. The idea of using overlapping sets to evaluate curvature information helps if a node failure is encountered. However, the size of overlapping is large, reshuffling and redistributing data among nodes may need costly communications. Progressive batching L-BFGS \cite{bollapragada2018progressive} gradually increases the batch size instead of using the full batch in the multi-batch scheme to improve the computational efficiency. Another line of work explores the decentralized  quasi-Newton methods (Decentralized BFGS) \cite{eisen2016decentralized,eisen2017decentralized}.  Taking physical network structure into consideration, the nodes communicate only with their neighbors and perform local computations. 
Recently,  the distributed averaged quasi-Newton methods and the adaptive distributed variants (DAve-QN) have been implemented in \cite{soori2020dave}, which can be very efficient for low dimensional problems because that they need to read and write the whole Hessian matrix during updating. Notice that in the analysis of DAve-QN method, they have a strong requirement for condition number of the objective functions in Lemma 2 of \cite{soori2020dave}, which is hard to satisfied in most real datasets. However, when dealing with machine learning problems under big datasets, or considering the possibility of node failures, deterministic algorithms like the decentralized ones perform no better than stochastic ones. 

\subsection{Contributions} In this paper, we present not only parallel but an asynchronous regime of SQN method with the VR technique, which we call AsySQN. The comparison of our method against existing methods is illustrated in Table \ref{table1}. AsySQN aims to design a stochastic and parallel process for every step during calculating the search directions. To the best of our knowledge, this is the first effort in implementing the SQN methods in an asynchronous parallel way instead of directly using some local acceleration techniques. The basic idea of asynchronous has been explored by first-order methods, but it is not straightforward to apply to stochastic quasi-Newton methods, and not even clear if such a procedure converges. We provide a theoretical guarantee of a linear convergence rate for the proposed algorithm. 
Notice that even though there is  another commonly used way of lifting the convergence rate, by exploring the self-concordant functions, we choose to use the VR technique so that we can have more general objective functions when conducting experiments.
In addition, we remove the requirement of gradients sparsity in previous methods \cite{recht2011hogwild,reddi2015variance}. Although we focus on the L-BFGS method, the proposed framework can be applied to other quasi-Newton methods including BFGS, DFP, and the Broyden family \cite{wright1999numerical}. In our empirical evaluation, we show that the proposed algorithm is more computational efficient than its non-parallel counterpart. Additional experiments are conducted to show that AsySQN maintains
the property of a quasi-Newton method that is to obtain a solution with high precision faster on ill-conditioned problems than first-order methods without zigzagging.

\subsection{Notations}
  The Euclidean norm of vector $x$ is denoted by $\|x\|_{2}$, and without loss of generality, we omit the subscript $2$ and use $\|x\|$. We use $x^{*}$ to denote the global optimal solution of Eq.(\ref{equ2}). For simplicity, we also use $f^{*}$ to be the corresponding optimal value $f(x^*)$. The notation $E[\cdot]$ means taking the expectation in terms of all random variables, and $I$ denotes the identity matrix. The condition number of a matrix $A$ is defined by $\kappa(A) = \frac{\sigma_{max}(A)}{\sigma_{min}(A)}$, where $\sigma_{max}(A)$ and $\sigma_{min}(A)$ denote, respectively, the largest and the smallest singular values of $A$. For square matrices $A$ and $B$ of the same size, $A \preceq B$, iff $B-A$ is positive semidefinite. Given two numbers $a \in \mathbb{Z}$ (integers), $b \in\mathbb{Z^{+}}$ (positive intergers), $mod (a,b) = 0$, if there exists an integer $k$ satisfying $a = kb$.

\section{Algorithm}
\label{sec:alg}

In this paper, we consider the following finite-sum problem :
\begin{equation}\label{equ1}
f(x):= \frac{1}{n}\sum_{i=1}^{n}f(x;z_{i})\stackrel{def}{=}\frac{1}{n}\sum_{i=1}^{n}f_{i}(x),
\end{equation}
where $\{z_{i}\} _{i=1}^{n}$ denote the training examples, $f: \mathbb{R}^{d}\rightarrow \mathbb{R}$ is a loss function that is parametrized by $x$ which is to be determined in the optimization process, and $f_{i}$ is the loss occurred on $z_i$. Most machine learning models (with model parameter $x$) can be obtained by solving the following minimization problem:
\begin{equation}\label{equ2}
\min_{x \in \mathbb{R}^{d}} f(x)=\frac{1}{n}\sum_{i=1}^{n}f_{i}(x).
\end{equation}
We aim to find an $\epsilon-$approximate solution, $x$, if $x$ satisfies:
\begin{equation}\label{equ3}
E[f(x)]-f(x^*)\leq\epsilon,
\end{equation}
where $f(x^*)$ is the global minimum if exists and $x^*$ is an optimal solution, $\epsilon$ is the targeted accuracy, and $E[f(x)]$ means to take the expectation of $f(x)$. We require a high precision solution, e.g., $\epsilon = 10^{-30}$.

\begin{algorithm}[t]
	\caption{AsySQN}
	\label{alg:framework}
	\begin{algorithmic}[1]
		\State {\bfseries Input:} Initial $w_{0}\in R^{d}$, step size $\eta \in R^{+}$, subsample size $b,b_{h}$, parameter $m$, $L$, $M$, and threads number $P$.
		\State Initialize $H_{0}=I$
		\State Initialize parameter $x$ in shared memory
		\For{$k=0,1,2,...$}	
				\State $w_k, \mu_{k} = ScheduleUpdate( x, k, m)$ (Algorithm \ref{alg:ScheduleUpdate})
		\For{$p = 0, 1, ... P-1$, \textbf{parallel}} 
		\For{$t=0$ {\bfseries to} $L-1$} 
		\State\label{Line11} Sample $S_{p,kL+t} \in \{1,2,..,n \}$
		\State Read $x$ from shared memory as $x_{p,kL+t}$
		\State Compute  $g_1 =\nabla f_{S_{p,kL+t}}(x_{p,kL+t})$
		\State\label{Line14} Compute  $g_2 =\nabla f_{S_{p,kL+t}}(x_{k})$
		\State Calculate the variance-reduced gradient: $v_{p,kL+t}=g_1-g_2+\mu_{k}$
		\If{$k \leq 1 $} 
		    \State  $x_{p,kL+t+1}= x_{p,kL+t}-\eta v_{p,kL+t}$
		\Else
	        \State  $x_{p, kL+t+1}=x_{p,kL+t}-\eta H_{k} v_{p,kL+t}$
	        \State where the search direction $p =- H_{k} v_{p,kL+t}$ is computed by the {\em Two-loop-recursion}($v_{p,kL+t}$,  $s_i$, $y_i$, $\forall i=k-M, \cdots, k-1$) (Algorithm \ref{alg:two-loop})
		\EndIf
		\State Write $x_{p, kL+t+1}$ to $x$ in shared memory
		\EndFor
		\EndFor
		\State\label{Line22} Sample $T_{k} \in \{1,2,..,n \}$
		\State\label{Line23} $x_{k}=\frac{1}{LP}\sum_{p=0}^{P-1}\sum_{i=kL}^{kL+L-1}x_{p,i}$
		\State\label{Line24} $s_{k}=x_{k}-x_{k-1}$
		\State\label{Line25} Option I: $y_{k}=\nabla f_{T_{k}}(x_{k})-\nabla f_{T_{k}}(x_{k-1} )$
		\State\label{Line26} Option II: $y_{k}=\nabla^{2}f_{T_{k}}(x_{k})s_{k} $
		\EndFor	
	\end{algorithmic}
\end{algorithm}
\begin{algorithm}[t]
	\caption{{\em Schedule-update} $(x,k,m)$}	\label{alg:ScheduleUpdate}
	\begin{algorithmic}[1]
		\State {\bfseries Input:}$ x, k, m$.
		\State {\bfseries Output:} $w_k, \mu_k$
		\If{ mod($k , m$) == 0}
		\State $w_k = x$
		\State $\mu_k = \nabla f( x)$
		\EndIf
	\end{algorithmic}
\end{algorithm}
\begin{algorithm}[t]
	\caption{{\em Two-loop-recursion} ($v$, $s_i$, $y_i$, $\forall i=1, \cdots, M$)}
	\label{alg:two-loop}
	\begin{algorithmic}[1]
		\State {\bfseries Input:} $v$, $s_i$, $y_i$ where $i=1, \cdots, M$.
		\State {\bfseries Output:} $p$ 
		\State Initialize  $p = - v$
		\For{$i = M, ..., 1$}
		\State $\alpha_i = \frac{s_i\cdot p}{s_i\cdot y_i}$
		\State $p = p - \alpha_i\cdot y_i$
		\EndFor
		\State $p = (\frac{s_{k-1}\cdot y_{k-1}}{y_{k-1}\cdot y_{k-1}} )p$
		\For {$i = 1, ..., M$}
		\State $\beta= \frac{y_i\cdot p}{s_i\cdot y_i}$
		\State $p = p + (\alpha_i - \beta)\cdot s_i$
		\EndFor
	\end{algorithmic}
\end{algorithm}

\subsection{Review of the stochastic quasi-Newton method}
  According to the classical L-BFGS methods \cite{wright1999numerical}, given the current and last iterates $x_{k}$ and $x_{k-1}$ and their corresponding gradients, we define the correction pairs:
\begin{equation}\label{equ4}
s_{k}=x_{k}-x_{k-1},~~~ y_{k}=\nabla f_{k}-\nabla f_{k-1},
\end{equation}
and define $\rho_{k}=\frac{1}{y_{k}^{T}s_{k}}$, $V_{k}=I-\rho_{k}y_{k}s_{k}^{T}$.
Then the new Hessian inverse matrix can be uniquely updated based on the last estimate $H_{k-1}$:
\begin{equation}\label{equ5}
H_{k}=(I-\rho_{k}s_{k}y_{k}^{T})H_{k-1}(I-\rho_{k}y_{k}s_{k}^{T})+\rho_{k}s_{k}s_{k}^{T}
= V_k^{T} H_{k-1}V_k + \rho_{k}s_{k}s_{k}^{T} .
\end{equation}
When datasets are high dimensional, i.e. $d>0$ is a large number, and the Hessian and Hessian inverse matrices are $d\times d $ and maybe dense, storing and directly updating $H_{k}$ might be computationally prohibitive. Instead of exploring fully dense matrices, limited-memory methods only save the latest information (say, the recent $M$) of correction pairs $\{s_k,y_k\}$ to arrive a Hessian inverse approximation. Given the initial guess $H_{0}$ , which is usually set to an identity matrix, the new Hessian inverse at the $k$-th iteration can actually be approximated by repeated application of the formula (\ref{equ5}),
\begin{align}
H_{k} &=  (V_{k}^{T}\cdot V_{k-1}^{T}\cdot \dots \cdot V_{k-M+1}^{T}) H_{k-M} (V_{k-M+1}\cdot  V_{k-M+2}\cdot \dots \cdot V_{k})\nonumber\\
&+\rho_{k-M+1}(V_{k}^{T}\cdot V_{k-1}^{T}\cdot \dots \cdot V_{k-M}^{T})s_{k-M+1} \nonumber \\
& \cdot s_{k-M+1}^{T}(V_{k-M}\cdot  V_{k-M+1}\cdot \dots\cdot V_{k})\nonumber\\
&+ \dots\nonumber\\
& + \rho_{k}s_{k}s_{k}^{T}.
\end{align}

Based on this updating rule, a two-loop recursion procedure has been derived: the first for-loop multiplying all the left-hand side vectors and the second for-loop multiplying the right-hand side ones. In practice, if $v$ denotes the current variance-reduced gradient, we directly calculate Hessian-vector product $H\cdot v$ to find the search direction rather than actually computing and storing $H$. Details are referred to Algorithm \ref{alg:two-loop}.

\subsection{The proposed quasi-Newton scheme}

Algorithm \ref{alg:framework} outlines the main steps of our AsySQN method where $k$ indexes the epochs, $p$ indexes the processors, and $t$ indexes the iterations in each epoch within each processor. Following the same sampling strategy in \cite{byrd2016stochastic}, we use two separate subsets $S$ and $T$ of data to estimate the gradient  $\nabla f_S$ (Lines 8-11) and the Hessian inverse $H_T$  (Lines 22-24), respectively. At each iteration in an epoch, $S$ is used to update the gradient whereas the Hessian inverse is updated at each epoch using a subsample $T$. 
In order to better preserve the curvature information, the size of subsample $T$ should be relatively large, whereas the subsample $S$ should be smaller to reduce the overall computation.

\begin{figure*}[ht]
	\begin{center}
		\includegraphics[width=\linewidth]{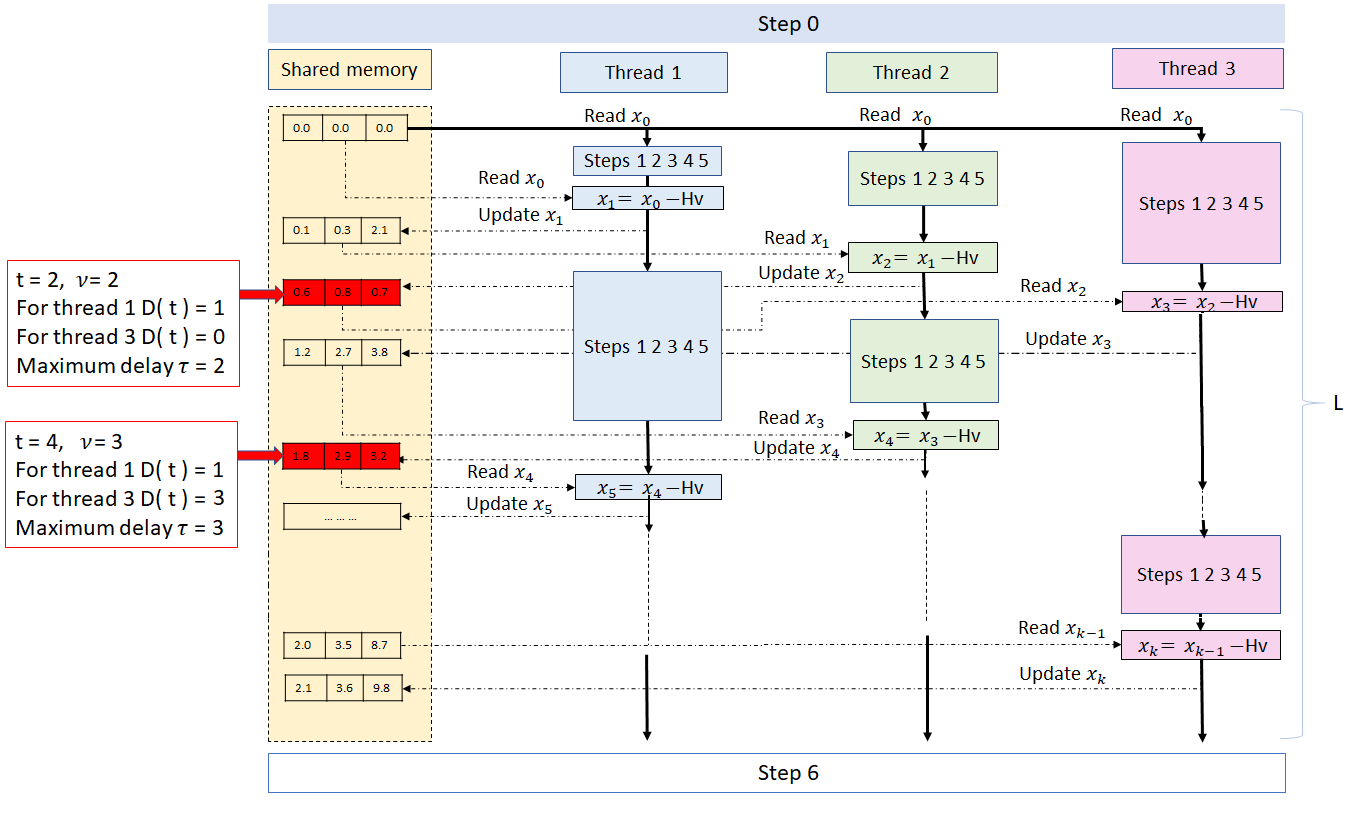}
		\caption{The key asynchronous parallel steps have been identified as follows: \textbf{1.} Subsample $S$; \textbf{2.} Read $x$ from the shared memory; \textbf{3.} Compute stochastic gradients $g_1$, $g_2$, and $v$; \textbf{4.} Compute the two-loop recursion and update $x$; \textbf{5.} Write $x$ to the shared memory; \textbf{6.} Subsample $T$ and update correction pairs. Before these steps, there is another step that computes the full gradient for all threads, i.e., \textbf{0.} Calculate full gradient. Here we show the asynchronous parallel updating of the threads where $t$ is a global timer in the epoch, recording the times that $x$ has been updated. For example, at $t = 2$, Thread $2$ updates $x$ to $x_2$ in the shared memory while Thread 1 and Thread 3 are in the middle of calculation using $x_1$ and $x_0$ respectively; hence the current $x$ index for Thread 1 is $1$ which we denote as $D(t)$ for Thread $1$, and similarly $D(t) = 0$ for Thread $2$, the maximum delay $\tau = t-D(t)$ is $2$. For Thread $2$, the delayed searching direction $H(x_{D'(2)})v(x_{D'(2)})$ is used to update $x_1$ to $x_2$, where $D'_2 = 0$; therefore the delay $\nu = 2$. Similar analysis can be applied to each state of $x$.  }
	\end{center}\label{fig1}
\end{figure*}

The algorithm requires several pre-defined parameters: step size $\eta$, limited memory size $M$, integers $m$ and $L$, the number of threads $P$ and batch size $|S|=b$, and $|T|=b_{h}$. We need the parameter $m$ because we only compute the full gradient of the objective function at the iterate $x_k$ every $m=\frac{n}{b\times L \times P}$ epochs, which will help reduce the computational cost. Algorithm \ref{alg:framework} can be split into hierarchically two parts: an inner loop (iteration) and an outer loop (epoch). The inner loop runs asynchronous parallel processes where each thread updates $x$ concurrently and the epoch loop includes the $L$ inner loops and an update on the curvature information. 

In our implementation, at the $k$-th epoch, the vector $s_{k}$ (Line 24) is the difference between the two iterates obtained before and after $L$ (parallel) inner iterations, where $x_{k}$ (Line 23) can be the average or the latest iterate in every $L$ iterations. The vector $y_{k}$ represents the difference between two gradients computed at $x_{k}$ and $x_{k-1}$ using the $T_{k}$ data subsample: $y_{k}=\nabla f_{T_{k}}(x_{k})-\nabla f_{T_{k}}(x_{k-1})$. By the mid-point theorem, there exists a $\xi \in [x_{k-1}, x_{k}]$, such that $y_{k}=\nabla^{2}f_{T_{k}}(\xi)\cdot s_{k}$. The two options (Line 25 and Line 26) of computing $y_k$ are both frequently used in quasi-Newton methods. In our numerical experiments, we run both options for comparisons to show that option II (Line 26) is more stable when the length of $s_k$, $\|s_{k}\|$, becomes small.

Stochastic gradient descent methods are generally difficult in achieving a solution of high precision, but they can reduce the objective function quickly in the beginning. Thus, a desirable way is to quickly find the optimal interval using a first-order (gradient-based) method and then switch to a second-order method. In practice, we take a warm start by the SVRG (not shown in Algorithm \ref{alg:framework}), and after certain iterations of quick descending, we then switch to the SQN method to efficiently get a more accurate optimal value.

\subsection{The proposed asynchronous parallel quasi-Newton}
 We propose a multi-threaded process in which each thread makes stochastic updates to a centrally stored vector $x$ (stored in a shared memory) without waiting, which is similar to the processes used in the asynchronous stochastic gradient descent (AsySGD) \cite{recht2011hogwild,lian2015asynchronous}, asynchronous stochastic  coordinate descent (AsySCD) \cite{liu2015asynchronous,hsieh2015passcode} and asynchronous SVRG (AsySVRG) \cite{reddi2015variance,zhao2016fast}. In multi-threaded process, there are $P$ threads, and each thread reads the latest $x$ from the shared memory, calculates the next iterate concurrently, and then write the iterate into the shared memory.

We use an example in Figure 1 to illustrate the process of this parallel algorithm. At the beginning of each epoch, the initial $x_0$ (for notational convenience) is the current $x_k$ from the shared memory.  When a thread finishes computing the next iterate, it will immediately write $x$ into the shared memory. Using the modern memory techniques, there are locks to guarantee that each time only one thread can write to the shared memory. Let us use a global timer $t$ that records any time when the shared memory is updated. In an asynchronous algorithm, when one thread is updating $x$ to $x_t$, another thread may still use an old state $x_{D(t )}$ to compute the search direction and update $x$. Let $D( t )\in [t]$ be the particular state of $x$ for a thread when the shared memory ($x$) is updated at time $t$. Specifically, let $D'( t )\in [t]$ be the state of $x$ that has been used to update $x$ to $x_t$. We assume that the maximum delay among processes is bounded by $\tau$, that is $t- D( t) \leqslant \tau$.  When Thread $p \in \{1,...,P\}$ updates $x$ to $x_t $, within this thread, the actual state of $x $ that has been used to calculate the search direction is $D'( t )$. Hence, we have  $x_{t} = x_{t -1} - \eta H(x_{D'(t)})v( x_{D'( t )})$, where the delay within Thread $p$ is $\nu = t - D'(t)$, obviously $\nu\leqslant\tau$. Considering that the maximum number of iterations in each epoch is $L$, we have $0 \leqslant \tau \leqslant L < m$.

In summary, we design an asynchronous parallel stochastic quasi-Newton method, named AsySQN, to improve the computational efficiency. In our experiments (Section 5), AsySQN shows obvious speedup compared with the original stochastic quasi-Newton methods SQN-VR due to the asynchronous parallel setting. Meanwhile, it still enjoys the advantages of common second-order methods: stability and accuracy. Furthermore, in some machine learning problems, it is not easy to do data normalization or even need more transformation to change the ill-conditioning behavior of objective function. In this case, AsySQN can perform better than AsySVRG or other first-order algorithms.

\section{Preliminaries}\label{preliminaries}
\subsection{Definitions}
  Each epoch consists of $m$ iterations, at the beginning of each epoch, suppose that the \textit{ScheduleUpdate} rule is triggered after each $m$ steps, and for brevity, we use $w_{k}$ to denote the iterations chosen at the $k$-th epoch.

At the $(k+1)-th$ epoch $t-th$ iteration:

The updating formula is: $x_{t+1}=x_{t}-\eta H_{t} v_{t}$, where step size $\eta\in \mathbb{R^{+}}$ determines how long should moves on, here we do not require diminishing way to guarantee convergence rate;

Stochastic Variance-Reduced gradient: $u_{t}=\nabla f_{i_{t}}(x_{t})-\nabla f_{i_{t}}(w_{k})+\nabla f(w_{k})$, where $i_{t}$ is randomly chosen from a subset in $S\subseteq \{1,..,n\}$. Easily, we get the expectation of stochastic Variance-Reduced gradient: $E[u_{t}]=\nabla f(x_{t})$;

The delayed stochastic Variance-Reduced (VR) gradient used in asynchronous setting: $v_{t}=\nabla f_{i_{t}}(x_{D(t)}) -\nabla f_{i_{t}}(w_{k})+\nabla f(w_{k})$. As the full gradient has been computed before, the gradient was estimated at $w_{k}$ whereas the current iteration should be $x_{D(t)}$ due to the delay. Taking expectation of the delayed gradient, we get: $E[v_{t}]=\nabla f(x_{D(t)})$.

\subsection{Assumptions}
  We make the same assumptions as non-asynchronous version of the SQN method that also uses VR technique \cite{moritz2016linearly}.

\begin{assumption}\label{hyp1}
The loss function $f_{i}$ is $\mu-$strongly convex,  that is,
\begin{displaymath}
f_{i}(y)\geq f_{i}(x)+<\nabla f_{i}(x), y-x>+\frac{\mu}{2}\|y-x\|^{2},\forall x, y.
\end{displaymath}
\end{assumption}

\begin{assumption}\label{hyp2}
The loss function $f_{i}$ has $l-$Lipschitz continuous gradients: 
\begin{displaymath}
\|\nabla f_{i}(x)-\nabla f_{i}(y)\|\leq l \|x-y\|, \forall x, y.
\end{displaymath}
\end{assumption}
Since the average of loss function $f_i$ over $i$ preserves the continuity and convexity, objective function $f$ also satisfies these two hypotheses, then we can easily derive the following lemmas.

\begin{lemma}\label{lemma1}
Suppose that Assumption \ref{hyp1} and \ref{hyp2} hold. The Hessian matrix will be bounded by two positive constants $\mu$ and $l$ such that 
\begin{displaymath}
\mu I \preceq B_{k} \preceq lI.
\end{displaymath}
\end{lemma}
The second-order information of the objective function $f$ is bounded by $ \mu$ and $l$ : $\mu I\leq\nabla^{2}f\leq l I$, therefore all Hessian matrices $B$ share the same bounds and generally we have the condition number of the objective function: $\kappa(B)=\frac{l}{\mu}\geq 1$.

\begin{lemma}\label{lemma2}
Suppose that Assumption \ref{hyp1} and \ref{hyp2} hold. Let $H_k$ be the Hessian inverse matrix. Then for all $k\geq 1$, there exist constants $0 < \mu_1 < \mu_2$ such that $H_k$ satisfies
\begin{displaymath}
\mu_{1} I \preceq H_{k} \preceq \mu_{2}I,
\end{displaymath}
where $\mu_{1}=\frac{1}{(d+M)l}$, $\mu_{2}=\frac{((d+M)l)^{d+M-1}}{\mu^{d+M}}$, $l$ is the Lipschitz constant, $M$ is the limited memory size, $d$ is the length of $x$. 
\end{lemma}

Proof can be referenced from \cite{moritz2016linearly}. Similarly, all generated Hessian inverse matrices have the same following condition number: $\kappa(H) = \frac{\mu_{2}}{\mu_{1}} = (d+M)^{d+M}(\kappa(B))^{d+M}$. Obviously, the ill-conditioning degree of Hessian inverse will surely be amplified. For large scale datasets, the shape of objective function can be extreme ill-conditioned. Thus, how to achieve a fair performance in this ill-conditioned situation deserves our attention. As a second-order method, our AsySQN still enjoys fast convergence rate in ill-conditioned situation (see the main theorem \ref{main theorem} and Corrollary \ref{cor}). We also do extensive experiments showing this property in our simulation numerical analysis.

\section{Convergence Analysis}\label{convergence analysis}
  In the asynchronous parallel setting, all working nodes read and update without synchronization. This will cause delayed updates for some other threads. Hence, the stochastic variance-reduced gradients may not be computed using the latest iterate of $x$. Here, we use gradient $v_{t}$ defined above, then each thread uses each its own $v_{t}$ to update $x_{t+1} = x_{t} - \eta H_{t} v_{t}$.

\begin{lemma}\label{lemma3}
In one epoch, the delay of parameter $x$ reading from the shared memory can be bounded by:
\begin{displaymath}
E[\|x_{t}-x_{D(t)}\|]\leq\eta\mu_{2}\sum_{j=D(t)}^{t-1}E[\|v_{j}\|].	
\end{displaymath}
\end{lemma}

\begin{proof}
From Lemma \ref{lemma2} and updating formula  $x_{t+1} = x_{t} - \eta H_{t} v_{t}$,
\begin{align}
	E[\|x_{t}-x_{D(t)}\|]&\leq\sum_{j=D(t)}^{t-1}E[\|x_{j+1}-x_{j}\|]\nonumber\\
	&=\sum_{j=D(t)}^{t-1}\eta E[\| H_{t} v_{t} \|]\nonumber\\
	&\leq\eta\mu_{2}\sum_{j=D(t)}^{t-1}E[\|v_{j}\|].\nonumber	
\end{align}
\end{proof}

\begin{lemma}\label{lemma4}
Suppose that Assumption \ref{hyp2} holds, the delay in the stochastic VR gradient of the objective function $f:$ 
\begin{displaymath}
E[\|u_{t}-v_{t}\|]\leq l \eta \mu_{2}\sum_{j=D(t)}^{t-1}E[\|v_{j}\|].
\end{displaymath}
\end{lemma}

\begin{proof}
From Assumption \ref{hyp2}, Lemma \ref{lemma2} and \ref{lemma3},
\begin{align}
	E[\|u_{t}-v_{t}\|] &= E[\|\nabla f(x_{t}) -\nabla f(x_{D(t)})\|]\nonumber\\
	&\leq l E[\|x_{t}-x_{D(t)}\|] \nonumber\\
	&\leq l \eta \mu_{2}\sum_{j=D(t)}^{t-1}E[\|v_{j}\|].\nonumber
\end{align}
\end{proof}

\begin{lemma}\label{lemma5}
	Suppose that Assumption \ref{hyp1} holds, denote $w^{*}$ to be the unique minimizer of objective function $f$. Then for any $x$, we have
\begin{displaymath}
	\|\nabla f(x)\|^{2}\geq 2\mu(f(x)-f(w^{*}))
\end{displaymath}
\end{lemma}

\begin{proof}
	From Assumption \ref{hyp1},
\begin{align}
	f(w^{*}) &\geq f(x)+\nabla f(x)^{T}(w^{*}-x)+\frac{\mu}{2}\|w^{*}-x\|^{2}\nonumber\\
	&\geq f(x)+\min_{\xi}(\nabla f(x)^{T}\xi+\frac{\mu}{2}\|\xi\|^{2})\nonumber\\
	&=f(x)-\frac{1}{2\mu}\|\nabla f(x)\|^{2}.\nonumber
\end{align}
Letting $\xi=w^{*}-x,$ when $\xi=-\frac{\nabla f(x)}{\mu}$, the quadratic function can achieve its minimum.
\end{proof}

\begin{lemma}\label{lemma6}
Suppose that Assumptions \ref{hyp1} and \ref{hyp2} hold. Let $f^{*}$ be the optimal value of Eq.(\ref{equ2}). Considering the stochastic variance-reduced gradient, we have:
\begin{displaymath}
E[\|u_{t}\|^{2}]\leq 4 l E[f(x_{t})-f^{*}+f(w_{k})-f^{*}].
\end{displaymath}
\end{lemma}

Proof reference \cite{johnson2013accelerating}.

\begin{theorem}\label{theorem1}
Suppose that Assumptions \ref{hyp1} and \ref{hyp2} hold, the gradient delay in an epoch can be measured in the following way:
\begin{displaymath}
	\sum_{t=km}^{km+m-1}E[\|v_{t}\|^{2}]\leq \frac{2}{1-2l^{2}\eta^{2}\mu_{2}^{2}\tau^{2}}\sum_{t=km}^{km+m-1}E[\|u_{t}\|^{2}].
\end{displaymath}	
\end{theorem}

\begin{proof}
\begin{align}
E[\|v_{t}\|^{2}] &\leq 2E[\|v_{t}-u_{t}\|^{2}] + 2E[\|u_{t}\|^{2}]\nonumber\\
&\leq 2l^{2} \eta^{2} \mu_{2}^{2}\tau\sum_{j=D(t)}^{t-1}E[\|v_{j}\|^{2}] + 2E[\|u_{t}\|^{2}]\nonumber
\end{align}

Summing from $t=km$ to $ t=km+m-1$, we get
\begin{align}
\sum_{t=km}^{km+m-1}E[\|v_{t}\|^{2}]
&\leq\sum_{t=km}^{km+m-1} 2l^{2}\eta^{2}\mu_{2}^{2}\tau \sum_{j=D(t)}^{t-1}E[\|v_{j}\|^{2}] + 2\sum_{t=km}^{km+m-1}E[\|u_{t}\|^{2}]\nonumber\\
&\leq 2l^{2}\eta^{2}\mu_{2}^{2}\tau^{2} \sum_{t=km}^{km+m-1}E[\|v_{t}\|^{2}] + 2\sum_{t=km}^{km+m-1}E[\|u_{t}\|^{2}]\nonumber
\end{align}
	By rearranging, we get the inequality above.
\end{proof}

\begin{theorem}[Main theorem]\label{main theorem}
Suppose that Assumptions \ref{hyp1} and \ref{hyp2} hold, step size $\eta$ and epoch size $m$ are chosen such that the following condition holds:
\begin{displaymath}
	0 < \theta:=\frac{1+C}
	{1+m\eta\mu\mu_{1}-C} <1, 
\end{displaymath}	
where $~C=\frac{4ml^{2}\eta^{2}\mu_{2}^{2}(l\eta\mu_{1}\tau+1)}{1-2l^{2}\eta^{2}\mu_{2}^{2}\tau^{2}}.$
Then for all $k\geq 0,$ we have
\begin{displaymath}
	E[f(w_{k+1})-f^{*}]\leq\theta E[f(w_{k})-f^{*}].
\end{displaymath}	
\end{theorem}

\begin{proof}
Let's start from the Lipschitz continuous condition:
\begin{align}
f(x_{t+1})&\leq f(x_{t})+\nabla f(x_{t})(x_{t+1}-x_{t})+\frac{l}{2}\|x_{t+1}-x_{t}\|^{2}\nonumber\\
&=f(x_{t})-\eta H_{r}\nabla f(x_{t})v_{t}+\frac{l\eta^{2}H_{r}^{2}}{2}\|v_{t}\|^{2}\nonumber\\
&\leq f(x_{t})-\eta\mu_{1}\nabla f(x_{t})v_{t}+\frac{l\eta^{2}\mu_{2}^{2}}{2}\|v_{t}\|^{2}\nonumber
\end{align}
Taking expectation at both hands,
\begin{displaymath}
E[f(x_{t+1})]\leq E[f(x_{t})]-\eta\mu_{1}E[u_{t}]E[v_{t}]+\frac{l\eta^{2}\mu_{2}^{2}}{2}E[\|v_{t}\|^{2}]
\end{displaymath}
Given $ab=\frac{(a+b)^{2}-(a-b)^{2}}{4}$, we can derive: $ -ab=\frac{1}{2}(a-b)^{2}-\frac{1}{2}a^{2}-\frac{1}{2}b^{2}$. Thus,
\begin{align}
&-E[u_{t}]E[v_{t}]\nonumber \\
&=\frac{1}{2}(E[u_{t}]-E[v_{t}])^{2}-\frac{1}{2}(E[u_{t}])^{2}-\frac{1}{2}(E[v_{t}])^{2}\nonumber\\
&= \frac{1}{2}(\nabla f(x_{t})-\nabla f(x_{D(t)}))^{2}-\frac{1}{2}(\nabla f(x_{t}))^{2}-\frac{1}{2}(\nabla f(x_{D(t)}))^{2}\nonumber\\
&\leq\frac{l^{2}}{2}\|x_{t}-x_{D(t)}\|^{2}-\mu (f(x_{t})-f^{*})-\mu (f(x_{D(t)})-f^{*})\nonumber\\
&=\frac{l^{2}\eta^{2}\mu_{2}^{2}}{2}\sum_{j=D(t)}^{t-1}E[\|v_{j}\|^{2}]-\mu (f(x_{t})-f^{*})-\mu (f(x_{D(t)})-f^{*})\nonumber
\end{align}
The first inequality comes from Lipschitz continuous condition and Lemma \ref{lemma5}.

Substituting this in the inequality above, we get the following result:	
\begin{align}
&E[f(x_{t+1})]\nonumber \\
&\leq E[f(x_{t})]+\eta\mu_{1}(\frac{l^{2}\eta^{2}\mu_{2}^{2}}{2}\sum_{j=D(t)}^{t-1}E[\|v_{j}\|^{2}]-\mu (f(x_{t})-f^{*})\nonumber\\
&-\mu (f(x_{D(t)})-f^{*}))+\frac{l\eta^{2}\mu_{2}^{2}}{2}E[\|v_{t}\|^{2}]\nonumber\\
&\leq E[f(x_{t})]+\frac{l^{2}\eta^{3}\mu_{2}^{2}\mu_{1}}{2}\sum_{j=D(t)}^{t-1}E[\|v_{j}\|^{2}]-\eta\mu\mu_{1}E[f(x_{t})-f^{*}]\nonumber\\
&-\eta\mu\mu_{1}E[f(x_{D(t)})-f^{*}]+\frac{l\eta^{2}\mu_{2}^{2}}{2}E[\|v_{t}\|^{2}]\nonumber\\
&\leq E[f(x_{t})]+\frac{l^{2}\eta^{3}\mu_{2}^{2}\mu_{1}}{2}\sum_{j=D(t)}^{t-1}E[\|v_{j}\|^{2}]
-\eta\mu\mu_{1}E[f(x_{t})-f^{*}]\nonumber\\
&+\frac{l\eta^{2}\mu_{2}^{2}}{2}E[\|v_{t}\|^{2}]\nonumber
\end{align}
The last inequality comes from the fact that $f^{*}=min f(x), ~~f(x_{D(t)})-f^{*}\geq 0. $
	
Summing from $t=km$ to $ t=km+m-1$, we get
\begin{align}
&\sum_{t=km}^{km+m-1}E[f(x_{t+1})-f(x_{t})]\nonumber\\
&\leq\frac{l^{2}\eta^{3}\mu_{2}^{2}\mu_{1}}{2}\sum_{t=km}^{km+m-1}\sum_{j=D(t)}^{t-1}E[\|v_{j}\|^{2}]+\frac{l\eta^{2}\mu_{2}^{2}}{2}\sum_{t=km}^{km+m-1}E[\|v_{t}\|^{2}]\nonumber\\
&-\eta\mu\mu_{1}\sum_{t=km}^{km+m-1}E[f(x_{t})-f^{*}]\nonumber\\
&\leq\frac{l^{2}\eta^{3}\mu_{2}^{2}\mu_{1}\tau}{2}\sum_{t=km}^{km+m-1}E[\|v_{t}\|^{2}]+\frac{l\eta^{2}\mu_{2}^{2}}{2}\sum_{t=km}^{km+m-1}E[\|v_{t}\|^{2}]\nonumber\\
&-\eta\mu\mu_{1}mE[f(w_{k+1})-f^{*}]\nonumber\\
\end{align}
\begin{align}
&=\frac{l\eta^{2}\mu_{2}^{2}(l\eta\mu_{1}\tau+1)}{2}\sum_{t=km}^{km+m-1}E[\|v_{t}\|^{2}]-\eta\mu\mu_{1}mE[f(w_{k+1})-f^{*}]\nonumber\\
&\leq\frac{l\eta^{2}\mu_{2}^{2}(l\eta\mu_{1}\tau+1)}{2}\cdot\frac{2}{1-2l^{2}\eta^{2}\mu_{2}^{2}\tau^{2}}\sum_{t=km}^{km+m-1}E[\|u_{t}\|^{2}] \nonumber\\
&-\eta\mu\mu_{1}mE[f(w_{k+1})-f^{*}]\nonumber\\
&=\frac{l\eta^{2}\mu_{2}^{2}(l\eta\mu_{1}\tau+1)}{1-2l^{2}\eta^{2}\mu_{2}^{2}\tau^{2}}\sum_{t=km}^{km+m-1}E[\|u_{t}\|^{2}]-\eta\mu\mu_{1}mE[f(w_{k+1})-f^{*}]\nonumber\\
&\leq\frac{l\eta^{2}\mu_{2}^{2}(l\eta\mu_{1}\tau+1)}{1-2l^{2}\eta^{2}\mu_{2}^{2}\tau^{2}}\sum_{t=km}^{km+m-1}4lE[f(x_{t})-f^{*}+f(w_{k})-f^{*}]\nonumber\\
&-\eta\mu\mu_{1}mE[f(w_{k+1})-f^{*}]\nonumber\\
&\leq(\frac{4ml^{2}\eta^{2}\mu_{2}^{2}(l\eta\mu_{1}\tau+1)}{1-2l^{2}\eta^{2}\mu_{2}^{2}\tau^{2}}-m\eta\mu\mu_{1})E[f(w_{k+1})-f^{*}]\nonumber\\
&+\frac{4l^{2}\eta^{2}\mu_{2}^{2}(l\eta\mu_{1}\tau+1)}{1-2l^{2}\eta^{2}\mu_{2}^{2}\tau^{2}}E[f(w_{k})-f^{*}]\nonumber
\end{align}
The second inequality uses the assumption that we have a delay $\tau $ in one epoch iteration. The third inequality uses the results of Theorem \ref{theorem1}, hence all delayed stochastic VR gradients can be bounded by the stochastic VR gradients. The fourth inequality holds for Lemma \ref{lemma6}. In the epoch-based setting,  
\begin{align*}
\sum_{t=km}^{km+m-1}E[f(x_{t+1})-f(x_{t})]&= E[f(x_{(k+1)m})-f(x_{km})]\\
&=E[f(w_{k+1})-f(w_{k})].
\end{align*}
Thus the above inequality gives
\begin{align}
&E[f(w_{k+1})-f^{*}]\nonumber \\
&\leq E[f(w_{k})-f^{*}]+\frac{4l^{2}\eta^{2}\mu_{2}^{2}(l\eta\mu_{1}\tau+1)}{1-2l^{2}\eta^{2}\mu_{2}^{2}\tau^{2}}\sum_{t=km}^{km+m-1}E[f(w_{k})-f^{*}]\nonumber\\
&-(m\eta\mu\mu_{1}-\frac{4ml^{2}\eta^{2}\mu_{2}^{2}(l\eta\mu_{1}\tau+1)}{1-2l^{2}\eta^{2}\mu_{2}^{2}\tau^{2}})E[f(w_{k+1})-f^{*}].\nonumber
\end{align}
By rearranging the above gives
\begin{align}
(1+m\eta\mu\mu_{1}-\frac{4ml^{2}\eta^{2}\mu_{2}^{2}(l\eta\mu_{1}\tau+1)}{1-2l^{2}\eta^{2}\mu_{2}^{2}\tau^{2}})E[f(w_{k+1})-f^{*}]\nonumber\\
\leq(1+\frac{4ml^{2}\eta^{2}\mu_{2}^{2}(l\eta\mu_{1}\tau+1)}{1-2l^{2}\eta^{2}\mu_{2}^{2}\tau^{2}})E[f(w_{k})-f^{*}].\nonumber
\end{align}
Hence, we get the following result:
\begin{displaymath}
E[f(w_{k+1}-f^{*})]\leq\theta E[f(w_{k}-f^{*})],
\end{displaymath}
where $\theta=\frac{1+C}{1+m\eta\mu\mu_{1}-C},~C=\frac{4ml^{2}\eta^{2}\mu_{2}^{2}(l\eta\mu_{1}\tau+1)}{1-2l^{2}\eta^{2}\mu_{2}^{2}\tau^{2}}.$  
\end{proof}

From Theorem \ref{main theorem} we can see that $\eta, \tau$ should satisfy condition $1-2l^{2}\eta^{2}\mu_{2}^{2}\tau^{2}>0$ to ensure the result holds. Therefore, we reach an upper bound of the step size: $\eta<\frac{1}{\sqrt{2}l\mu_{2}\tau}$. By substituting the result of Lemma \ref{lemma2}, we can easily derive  $l\mu_{2}\tau=l\frac{((d+M)l)^{d+M-1}}{\mu^{d+M}}\tau=(d+M)^{d+M-1}(\frac{l}{\mu})^{d+M}\tau>1$.

To ensure that $\theta=\frac{1+C}{1+m\eta\mu\mu_{1}-C} \in (0,1)$, we require $m\eta\mu\mu_{1}>2C$, where $C > 0$, and this requirement can be written equivalently in the following way:
\[
(1-2l^{2}\eta^{2}\mu_{2}^{2}\tau^{2})m\eta\mu\mu_{1} > 8ml^{2}\eta^{2}\mu_{2}^{2}(l\eta\mu_{1}\tau+1).
\]
We can treat the above inequality as quadratic equation with respect to step size $\eta$,
\[
2l^{2}\mu_{1}\mu_{2}^{2}\tau(4l+\mu\tau)\eta^{2}+8l^{2}\mu_{2}^{2}\eta-\mu\mu_{1} <0,
\]
and it can be checked that when $\eta \in (0, \frac{\sqrt{16l^{2}\mu^{2}_{2}+2\mu_{1}^{2}\mu\tau(4l+\mu\tau )}-4l\mu_{2}}{2l\mu_{1}\mu_{2}\tau(4l+\mu\tau)})$, the inequality will be satisfied. 
Noticing that this range falls into the previous restriction $\eta<\frac{1}{\sqrt{2}l\mu_{2}\tau}$, then we get a more accurate bound of step size to ensure the convergence rate $\theta$ falls in range $(0,1)$. For simplicity, we set the step size parameter $\eta=\frac{1}{2l\mu_2 m}$, it is easy to check that the quadractic inequality above will be certainly satisfied.

\begin{corollary} \label{cor}
Suppose that the conditions in Theorem \ref{main theorem} hold and for simplicity we set parameters as step size $\eta=\frac{1}{2l\mu_2 m}$ and the maximum delay parameter $\tau = L < m$. If $m>8\kappa(B)\kappa(H)+4\kappa(B)$, then the Asynchronous parallel SQN will have the following convergence rate:
\begin{displaymath}
	E[f(w_{k+1})-f^{*}]\leq \frac{(m+2)\kappa(B)\kappa(H)+\kappa (B)}{(m-2)\kappa(B)\kappa(H)-\kappa (B)+\frac{m}{2}} E[f(w_{k})-f^{*}].
\end{displaymath}	
\end{corollary}

\begin{proof}
We first check that when $\eta=\frac{1}{2l\mu_2 m}$, 
\begin{align}
C&=\frac{4ml^{2}\eta^{2}\mu_{2}^{2}(l\eta\mu_{1}\tau+1)}{1-2l^{2}\eta^{2}\mu_{2}^{2}\tau^{2}} \nonumber\\
&=\frac{\frac{\mu_1}{2\mu_2}\frac{L}{m} + 1}{m (1-\frac{1}{2}(\frac{L}{m})^2)}\nonumber\\
&< \frac{\frac{1}{\kappa(H)} + 2}{m}.\nonumber
\end{align}
Then we get a simplified $C$, which is still positive. The convergence rate will correspondingly become: 
\begin{align}
\theta &=\frac{1+C}{1+m \eta \mu \mu_{1} -C}\nonumber\\
&<\frac{1+\frac{\frac{1}{\kappa(H)} + 2}{m}}{1+\frac{1}{2\kappa(B)\kappa(H)} -\frac{\frac{1}{\kappa(H)} + 2}{m}}\nonumber\\
&=\frac{(m+2)\kappa(B)\kappa(H)+\kappa (B)}{(m-2)\kappa(B)\kappa(H)-\kappa (B)+\frac{m}{2}}.\nonumber
\end{align}
When $m>2,$ $\theta >0$ can be guaranteed. Further, to ensure convergence rate $\theta<1$, $m> 8\kappa(B)\kappa(H)+4\kappa(B)$.
\end{proof}

\begin{remark}[Step size]
	As the step size  $\eta$ diminishing to $0,$ $C \to 0,$ and $\theta \to 1$, which means the convergence rate becomes slower as the step size getting smaller. Decayed step size makes it easier to converge, however it is not the best choice for majority algorithms, and we use constant step size here.	
\end{remark}
	
It is well known that the deterministic quasi-Newton methods (e.g. L-BFGS) enjoy a superlinear convergence rate, while the stochastic version of quasi-Newton methods (including L-BFGS) will have a sublinear convergence rate in strongly convex optimization as a sacrifice; with the help of variance reduction technique, SQN-VR  reaches a linear convergence rate.  Compared with the original SQN-VR, even equipped with asynchronous setting, our method still achieves a linear convergence rate. The delay parameter $\tau$ in asynchronous setting affects the linear speed $\theta$ and more details are given as follows.

\begin{remark}[Asynchronous delay]
	If we assume that the delay parameter $\tau$ diminishes to $0,$ the convergence rate corresponds to the stochastic parallel one. SQN-VR \cite{moritz2016linearly} reaches a linear convergence rate, which is the upper bound since sequential algorithms have no delay in updating. Further, for synchronous parallel SQN, the convergence rate will not be changed.  However, once the delay is introduced by asynchronous, the convergence rate will certainly be impaired, even not converge. Our main theorem has proved that even with asynchronous updating scheme, the convergence rate still can reach linear. In the worst asynchronous case, i.e. the delay is maximized to be epoch size, Corollary \ref{cor} shows the linear convergence rate.
\end{remark}	

\begin{remark}[Affine invariance]
	It is known that second-order methods have the property that independent of linear scaling,  which is not true for gradient descent as it owns a rate with exponential relationship with condition number. For stochastic quasi-Newton methds, the convergence rate in references \cite{byrd2016stochastic,moritz2016linearly} did not show that SQN keeps the affine invariance. In our analysis, Theorem \ref{main theorem} shows that $\theta = O(\frac{1+C}{1+\frac{1}{\kappa(B)}-C})$, where $C =O(\frac{1}{\kappa(H)})$ when fixing all other parameters, specifically $m, d, M, \tau$. In the later Corollary \ref{cor} analysis, when $m$ is chosen large enough than $8\kappa(B)\kappa(H)+4\kappa(B)$, the linear convergence rate $\theta\approx \frac{m}{m+m/2}=\frac{2}{3}.$ Then we get the conclusion that $\theta$ doesn't depend on the condition number, our theoretical result verifies the affine invariance together with empirical study. 
\end{remark}

\begin{remark}[Limitation on subsample T]
	During the stochastic Hessian inverse construction, if the subsample $T$ goes to zero, no more curvature information will be captured, SQN will degrade to first-order gradient descent method; if $T$ goes to the entire data $n$, SQN has the whole costly Hessian inverse approximation, then SQN will not be so efficient. As a balance, subsample $T$ should be chosen as a moderately large, independent subset of data. In our experiments, we choose the size of $T$ to be 10 times the size of subsample $S$.
\end{remark}

\section{Empirical Study}\label{empirical}
  In this section, we compare the performance of our designed algorithm AsySQN with the state-of-the-art SQN-VR method  \cite{moritz2016linearly}  and distributed quasi-Newton method named DAve-QN \cite{soori2020dave} to show the enhancement raised by asynchronous parallel setting. More formally, we also make comparisons with classical stochastic gradient descent (SGD) \cite{zinkevich2010parallelized} and stochastic variance-reduced gradient method (SVRG) \cite{johnson2013accelerating} and their asynchronous versions \cite{recht2011hogwild,reddi2015variance} to give some insights of  first-order gradient methods and second-order quasi-Newton methods.
  We employ a machine with 32GB main memory and 8 Intel Xeon CPU E5-2667 v3 @ 3.2 GHz processors. The experiments have been performed on both synthetic and real-world datasets, with different numbers of data points and dimensions.

\textbf{Datapass analysis.} We use the number of times the whole dataset will be visited, or called datapass, as performance measure of the convergence speed.
This is because it is independent of the actual implementation of the algorithms, and has been a well-established convention in the literature on both stochastic first-order and second-order methods\cite{johnson2013accelerating,defazio2014saga,harikandeh2015stopwasting,byrd2016stochastic,moritz2016linearly,gower2016stochastic}.

For each epoch, the SVRG and AsySVRG algorithms visit the whole dataset twice (datapass $= 2$): one for updating parameter $x$ and the other is for calculating average gradient. For SQN-VR and AsySQN, each epoch will take $2 + \frac{b_H}{b\times L \times P}$ datapasses, where $P$ is the number of threads. The extra datapass of AsySQN is introduced by Hessian-vector product calculations.

\textbf{Speedup analysis.} We compare stochastic quasi-Newton methods (SQN-VR) and its asynchronous parallel version (AsySQN) in terms of speedup to show the advantage of taking asynchronous parallel strategy. 
Here we consider the stopping criteria: $f( w ) - f^* < \epsilon$, where $f^*$ has been calculated before, we require high-precision optimum as error $\epsilon = 10^{-30}$. The speedup results for simulation and real datasets are shown in Figures \ref{simulation_d20d200} and \ref{libsvm_data}, respectively.

As we require no more on sparsity and Hessian matrix is always dense, lock has been added when updating parameter $x$ in shared memory, that is the reason why we have extra waiting time and have not reached the ideal linear speedup. However, obvious speedup still can be captured. Also, we compare with the lock-version of asynchronous first-order methods for fairness.

\textbf{CPU time analysis.} We also include the real clock time when making comparisons within second-order methods, as shown in Figures \ref{simulation_d2_time} and \ref{libsvm_data_time}. The stopping criteria and precision settings are the same with speedup analysis.

\textbf{Experiment setup.}  We conduct two simulations (moderate-dimensional I and high-dimensional II under ill-conditioning case) using synthetic datasets and three real datasets (Real$\_$sim, MNIST, RCV1) evaluations to show the comparisons of our methods against competitors. For all our experiments, the batch size $b$ is set to be $ 5, 10, 20, 50 $ and the Hessian batch size $b_H$ is set to be $10 b$. The limited memory size $M = 10 $ in both SQN-VR and AsySQN. After testing experiments in SQN-VR, we select $P\times L $ as roughly a preferable constant for each dataset, which means a smaller $L$ is preferred when more threads are used. For all the algorithms, a constant step size is chosen via grid search. All experiments are initialized randomly except simulation I since we want to check the trace of algorithms in simulation I, see Figure \ref{trace_for_ill_condi_d2_2nd}. All experiments are performed with 8 cores in the parallel setting.

\begin{figure}[t]
	\begin{center}
		\centerline{\includegraphics[width=\columnwidth]{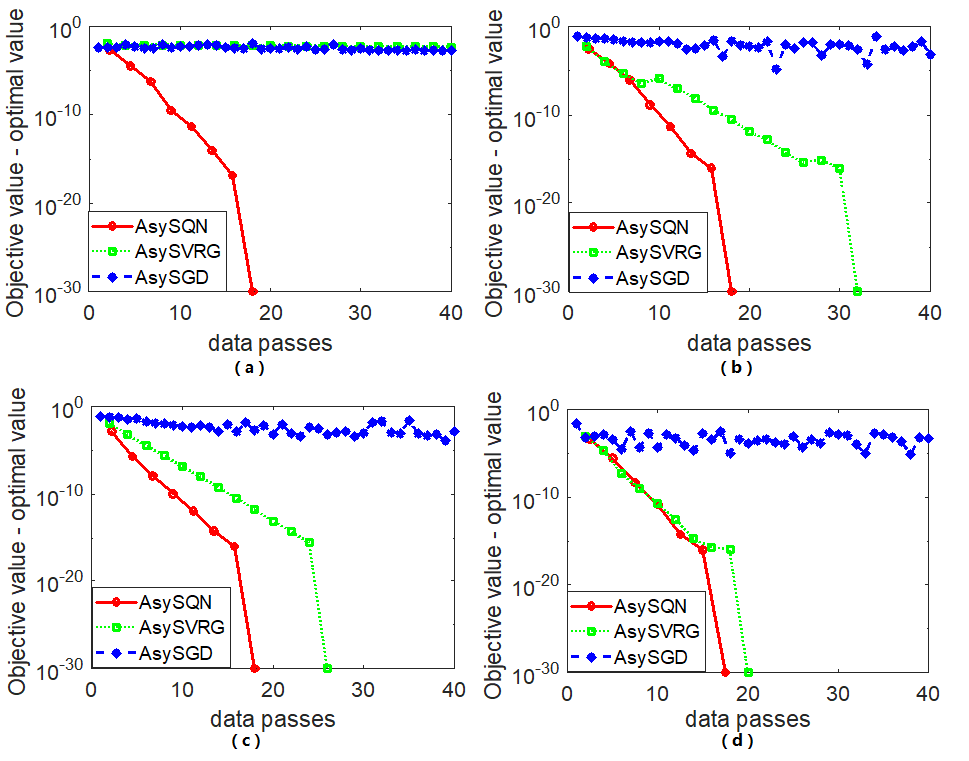}}
		\caption{Under various degrees of ill-conditioning least square problems: (a) extreme ill- conditioning; (b) ill-conditioning; (c) moderate ill-conditioning; (d) well-behaved.}
		\label{simulation_d2_datapass}
	\end{center}
\end{figure} 

\begin{figure}[t]
	\begin{center}
		\centerline{\includegraphics[width=\columnwidth]{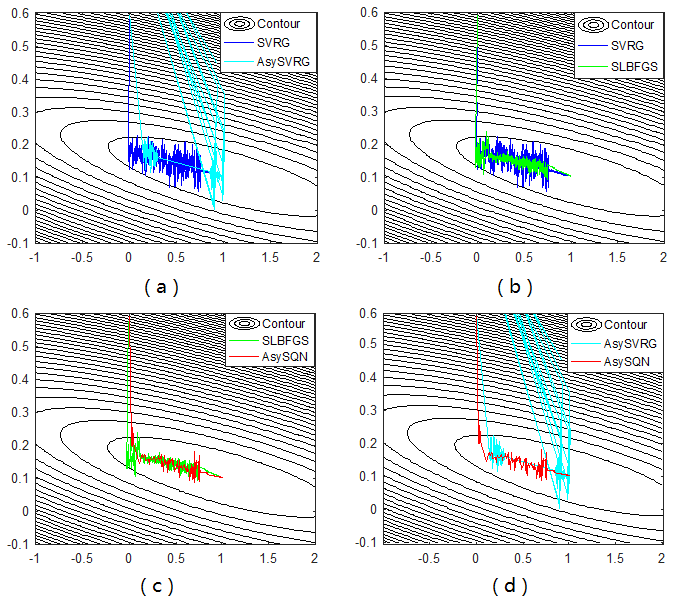}}
		\caption{Explore the results in ill-conditioning (same setting as Figure \ref{simulation_d2_datapass} (b)), demonstrates the stability of AsySQN. AsySQN can still quickly drop down to the optimum even though with little perturbation due to the asynchronous updates; while the AsySVRG methods have obvious jumping compared with SVRG due to the asynchronous setting.
		}
		\label{trace_for_ill_condi_d2_2nd}
	\end{center}
\end{figure}
\begin{figure}[t]
	\begin{center}
		\centerline{\includegraphics[width=\columnwidth]{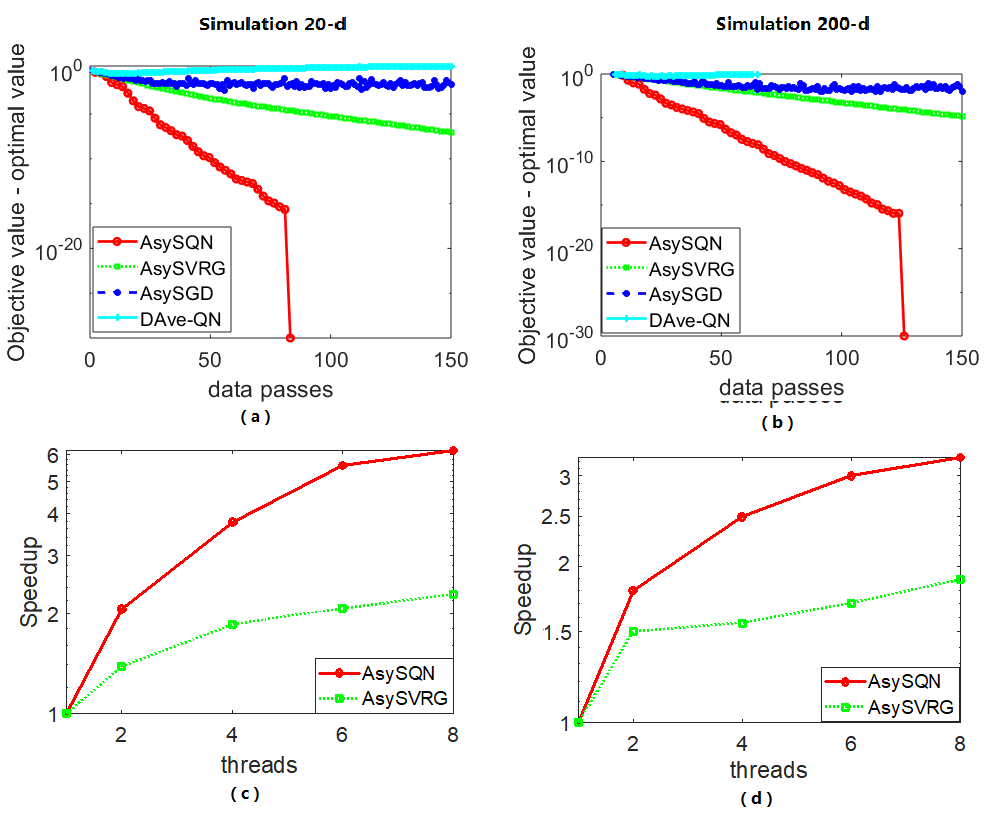}}
		\caption{Comparisons of different asynchronous algorithms in 20-D and 200-D simulation datasets.}
		\label{simulation_d20d200}
	\end{center}
\end{figure}

\begin{figure}[t]
	\begin{center}
		\centerline{\includegraphics[width=\columnwidth]{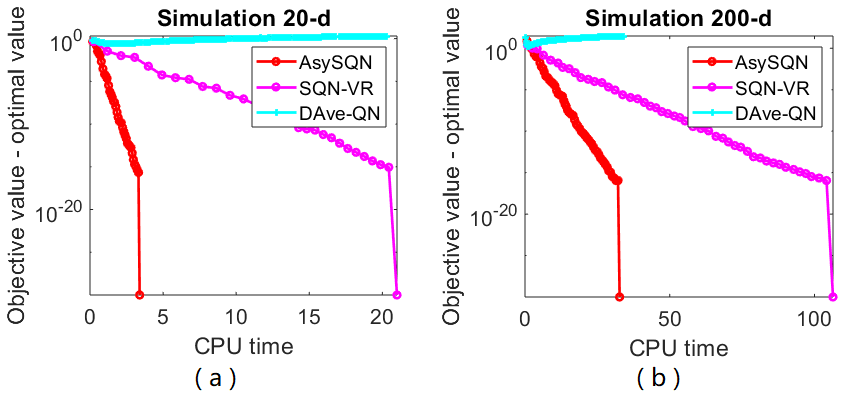}}
		\caption{CPU time of SQN-VR, its asynchronous version AsySQN, and distributed DAve-QN in 20-D and 200-D simulation datasets.}
		\label{simulation_d2_time}
	\end{center}
\end{figure}

\subsection{Experiments using synthetic datasets }

In order to explore some properties of our algorithm, simulation I and simulation II are performed for the least square problem:
\begin{equation}
\min_{x \in \mathbb{R}^{d}}\frac{1}{n}\sum_{i=1}^{n}(y_i-z_i^Tx)^2.
\end{equation}

The goal of simulation I is to study that under what circumstance our designed algorithm AsySQN will be preferred compared with first-order methods (AsySGD and AsySVRG). Various degrees of ill-conditioning least square problems are generated in the following procedures: $z_{1}$ and $z_{2}$ are two features generated uniformly from the interval $[0, 1]$. We then label $y = a\times z_{1} + b\times z_{2} + \epsilon $, where $ \epsilon$ satisfies normal distribution with mean $0$ and standard deviation $1$. To control the degree of ill-conditioning, we use four different settings of $(a,b)$ pairs: $(0.1, 10), (1, 10 ),  (1, 5)$ and $(1, 1)$, which correspond to extreme ill-conditioned, ill-conditioned, moderate ill-conditioned and well-behaved least squared problems, respectively.

\begin{figure*}
	\begin{center}
		\centerline{\includegraphics[width=0.8\linewidth]{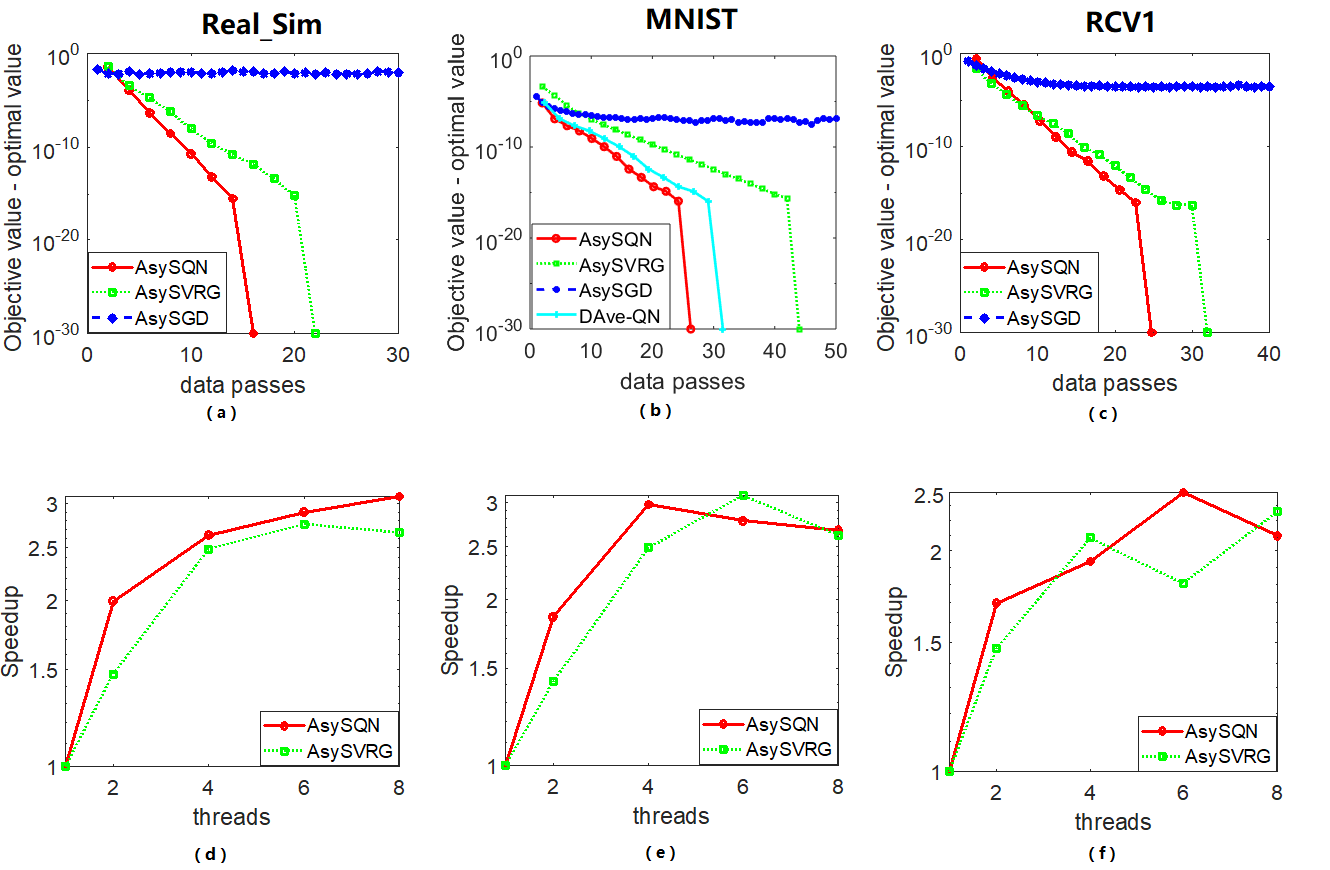}}
		\caption{Datapasses and speedup of logistic regression with Real$\_$sim dataset (a)(d) , MNIST dataset (b)(e) and SVM with RCV1 dataset (c)(f).}
		\label{libsvm_data}
	\end{center}
\end{figure*}

\begin{figure*}
	\begin{center}
		\centerline{\includegraphics[width=0.8\linewidth]{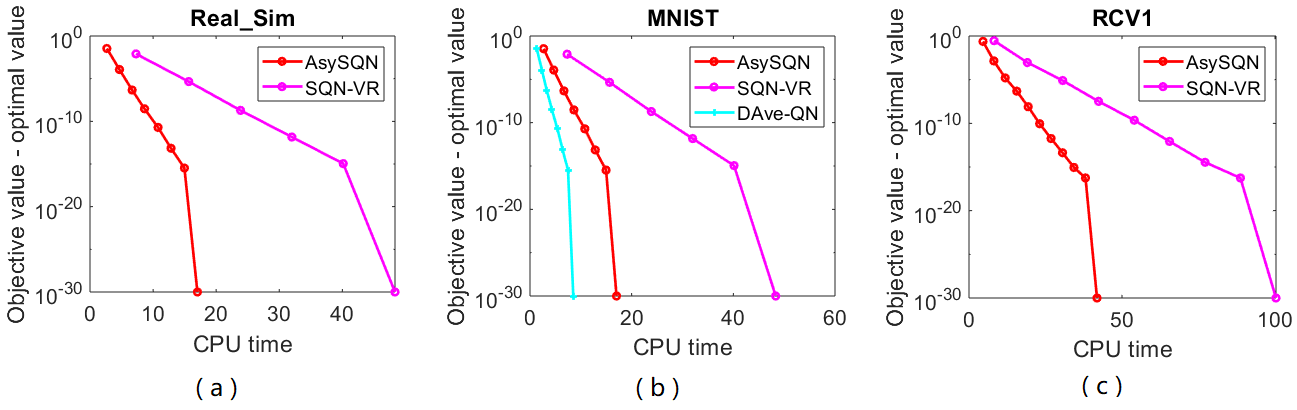}}
		\caption{CPU time of logistic regression with Real$\_$sim dataset (a), MNIST dataset (b) and SVM with RCV1 dataset (c).}
		\label{libsvm_data_time}
	\end{center}
\end{figure*}

From Figure \ref{simulation_d2_datapass}, we can see in extreme ill-conditioned case, AsySQN performs significantly better than first-order methods, it could achieve high precision solutions in short datapasses, where AsySVRG and AsySGD can barely solve the problem. In (moderate) ill-conditioned cases, AsySGD still performs bad, AsySVRG and AsySQN can linearly go to the optimum while AsySQN is faster in terms of datapasses. Even in well-behaved optimization, AsySQN is still comparable with AsySVRG. From (a)-(d), we can clearly conclude that the performance of AsySQN is more stable and verifies the affine invariance property.

To see the details, we plot the trace of AsySQN and AsySVRG respectively\footnote{AsySVRG is known to outperform AsySGD, and we don't include the details of AsySGD here.}.  Figure \ref{trace_for_ill_condi_d2_2nd} carefully explores the results showing in Figure \ref{simulation_d2_datapass}(b), in ill-conditioning optimization, how the trace information reveal the performance difference. 
We first study the well-known first-order method, SVRG, and its corresponding asynchronous version AsySVRG in Figure 3(a). It is clear to see that asynchronous setting results in wild jumps around the global minimum. However, second-order methods are well-known to be stable and we can see stochastic second-order method performs much more stable than SVRG in Figure 3(b). And even equipped with an asynchronous setting, AsySQN still exhibits good performance in terms of stability of convergence to the global optimal as shown in Figures 3(c,d).  In Figure 3(d), the side-by-side comparison between AsySVRG and AsySQN, AsySVRG has wild jumping around the global minimum.  In other words, AsySQN avoids the jumping which is a hard burden for lending asynchronous implementation directly to first-order methods.  Hence, we conclude that AsySQN outperforms AsySVRG under ill-conditioned circumstance.

In practice, it is appealing to solve machine learning problems over high dimensional datasets. Simulation II is performed where the dimensions of ill-conditioned simulation datasets are $(10^4, 20)$ and $(10^4, 200)$, respectively. AsySGD and AsySVRG are hard to quickly achieve a high precision solution, they cost almost 200 datapasses to approximate the optimal value within a precision of $10^{-4}$. DAve-QN method relies highly on the condition number of objective functions and it performs bad in our ill-conditioned simulations, which confirms their theoretical analysis. In contrast, AsySQN can achieve the high precision solution ($10^{-30}$) within finite datapasses, showing better performance with ill-conditioned datasets.  Further, AsySQN has an obvious speedup compared with synchronous SQN-VR methods in dealing with high dimensional data. See Figure \ref{simulation_d20d200}. We also include the speedup of AsySVRG as a reference.  When the number of threads is $1$, it is actually the original non-parallel version of stochastic L-BFGS method (SQN-VR in our experiment). The CPU time comparisons within second-order methods: the base line method SQN-VR, our designed method AsySQN, and the-state-of-the-art distributed quasi-Newton method DAve-QN, are shown in Figure \ref{simulation_d2_time}. AsySQN costs the least amount of time to reach a high precision solution, and shows obvious speedup compared with SQN-VR. In summary, we can get the conclusion that our designed algorithm scales very well with the problem size, and actually improve the speedup in dealing with high dimensional optimization problems.

\subsection{Experiments using real-world datasets}

  Three datasets have been used for real-dataset evaluations: Real$\_$sim, MNIST and RCV1, which can be downloaded from the LibSVM website\footnote{\url{http://www.csie.ntu.edu.tw/ cjlin/libsvmtools/datasets/}}.

We have conducted experiments on logistic regression as in Eq.(\ref{regression}) with real$\_$sim dataset (with $72,309$ data points of $20,958$ feature dimension, or $72,309\times 20,958$) and MNIST dataset ($60,000\times 780$):
\begin{equation}\label{regression}
\min_{x \in \mathbb{R}^{d}}\frac{1}{n}\sum_{i=1}^{n}(log (1+exp(y_{i}z_{i}^{T}x))+\lambda\|x\|^{2});
\end{equation}
and then we perform our algorithm on SVM as in Eq.(\ref{SVM}) with RCV1 dataset ($677,399\times 47,236$):
\begin{equation}\label{SVM}
\min_{x \in \mathbb{R}^{d}}\frac{1}{n}\sum_{i=1}^{n}( max(0, 1-y_{i}(z_{i}^{T}x))+\lambda\|x\|^{2});
\end{equation}
where $z_{i}\in \mathbb{R}^{d}$, and $y_{i}$ is the corresponding label. In our experiments, we set the regularizer $\lambda = 10^{-3}$.

Notice that we perform stochastic algorithms: AsySGD, AsySVRG, AsySQN and SQN-VR on all these real-datasets, while conducting DAve-QN method only on the smallest dataset, MNIST. This is because DAve-QN method needs to store and calculate hessian matrix (not uses limited memory version of BFGS method). Real$\_$sim and RCV1 have a large number of features, so it is hard to fulfill DAve-QN method on these two datasets. DAve-QN performs better in terms of CPU time over MNIST dataset. However, MNIST dataset has a moderate data size and is regular, that is, it doesn't have heterogeneity among features.

Figures \ref{libsvm_data} and \ref{libsvm_data_time} plot the comparisons of the methods mentioned above. By studying the convergence rate with respect to effective datapasses, AsySQN outperforms AsySVRG and AsySGD in all these real datasets. 
When the number of threads equals to 1, it is the original stochastic L-BFGS with VR method. With the increasing number of threads, the speedup will first show an ideal linear rate and ascend to a peak due to the asynchronous parallel within threads. When the number of threads reaches some marginal value, here in our experiments, the marginal value is particular 6, the speedup will descend. This is because the high cost of communication between multiple threads, the overhead of lock, and the resource contention, thus the growth curve of speedup is not as ideal as simulation plots. It may come down to the fact mentioned in earlier research \cite{xiao2014proximal} that benchmark downloaded from LibSVM has been row normalized. Hence AsySQN hasn't shown great advantage over AsySVRG but is still comparable in well-behaved case, which corresponds to the results of the simulation datasets above. In terms of speedup, we still plot AsySVRG as a reference: for first-order methods, locked AsySVRG reaches around 3 times speedup using 8 cores compared with SVRG \cite{reddi2015variance};  second-order method AsySQN, which requires more communications, can reach almost the same speedup with 8 cores. And we would recommend 6 cores in practice. In conclusion, our designed algorithm works well in real datasets.

\section{Conclusion}
\label{sec:conclusions}

  We have proposed a stochastic quasi-Newton method in an asynchronous parallel environment, which can be applied to solve large dense optimization with high accuracy. Compared with prior work on second-order methods, our algorithm represents the first effort in implementing the stochastic version in asynchronous parallel way and provides theoretical guarantee that a linear rate of convergence can be reached in the strongly convex setting. The experimental results demonstrate that our parallel version can effectively accelerate stochastic L-BFGS algorithm. For ill-conditioned problems, our algorithm carries the effectiveness of second-order methods and converges much faster than the best first-order methods such as AsySVRG. The theoretical proof of convergence rate can be used as framework in the asynchronous setting of stochastic second-order methods. In our future work, we may extend the algorithm and analysis to solve non-strongly convex or even non-convex loss functions.

\section*{Acknowledgments}

This work was funded by NSF grants CCF-1514357, IIS-1447711, and DBI-1356655 to Jinbo Bi.  Jinbo Bi was also supported by NIH grants K02-DA043063 and R01-DA037349.

\bibliography{SQNref}

\begin{thebibliography}{10}
\expandafter\ifx\csname url\endcsname\relax
  \def\url#1{\texttt{#1}}\fi
\expandafter\ifx\csname urlprefix\endcsname\relax\def\urlprefix{URL }\fi
\expandafter\ifx\csname href\endcsname\relax
  \def\href#1#2{#2} \def\path#1{#1}\fi

\bibitem{zinkevich2010parallelized}
M.~Zinkevich, M.~Weimer, L.~Li, A.~J. Smola, Parallelized stochastic gradient
  descent, in: Advances in neural information processing systems, 2010, pp.
  2595--2603.

\bibitem{recht2011hogwild}
B.~Recht, C.~Re, S.~Wright, F.~Niu, Hogwild: A lock-free approach to
  parallelizing stochastic gradient descent, in: Advances in neural information
  processing systems, 2011, pp. 693--701.

\bibitem{duchi2011adaptive}
J.~Duchi, E.~Hazan, Y.~Singer, Adaptive subgradient methods for online learning
  and stochastic optimization, Journal of Machine Learning Research 12~(Jul)
  (2011) 2121--2159.

\bibitem{defazio2014saga}
A.~Defazio, F.~Bach, S.~Lacoste-Julien, Saga: A fast incremental gradient
  method with support for non-strongly convex composite objectives, in:
  Advances in Neural Information Processing Systems, 2014, pp. 1646--1654.

\bibitem{kingma2014adam}
D.~P. Kingma, J.~Ba, Adam: A method for stochastic optimization, arXiv preprint
  arXiv:1412.6980.

\bibitem{reddi2015variance}
S.~J. Reddi, A.~Hefny, S.~Sra, B.~Poczos, A.~J. Smola, On variance reduction in
  stochastic gradient descent and its asynchronous variants, in: Advances in
  Neural Information Processing Systems, 2015, pp. 2647--2655.

\bibitem{schmidt2017minimizing}
M.~Schmidt, N.~Le~Roux, F.~Bach, Minimizing finite sums with the stochastic
  average gradient, Mathematical Programming 162~(1-2) (2017) 83--112.

\bibitem{dennis1977quasi}
J.~E. Dennis, Jr, J.~J. Mor{\'e}, Quasi-newton methods, motivation and theory,
  SIAM review 19~(1) (1977) 46--89.

\bibitem{nocedal1980updating}
J.~Nocedal, Updating quasi-newton matrices with limited storage, Mathematics of
  computation 35~(151) (1980) 773--782.

\bibitem{dembo1982inexact}
R.~S. Dembo, S.~C. Eisenstat, T.~Steihaug, Inexact newton methods, SIAM Journal
  on Numerical analysis 19~(2) (1982) 400--408.

\bibitem{liu1989limited}
D.~C. Liu, J.~Nocedal, On the limited memory bfgs method for large scale
  optimization, Mathematical programming 45~(1) (1989) 503--528.

\bibitem{bordes2009sgd}
A.~Bordes, L.~Bottou, P.~Gallinari, Sgd-qn: Careful quasi-newton stochastic
  gradient descent, Journal of Machine Learning Research 10~(Jul) (2009)
  1737--1754.

\bibitem{wang2017stochastic}
X.~Wang, S.~Ma, D.~Goldfarb, W.~Liu, Stochastic quasi-newton methods for
  nonconvex stochastic optimization, SIAM Journal on Optimization 27~(2) (2017)
  927--956.

\bibitem{mokhtari2018iqn}
A.~Mokhtari, M.~Eisen, A.~Ribeiro, Iqn: An incremental quasi-newton method with
  local superlinear convergence rate, SIAM Journal on Optimization 28~(2)
  (2018) 1670--1698.

\bibitem{bottou2018optimization}
L.~Bottou, F.~E. Curtis, J.~Nocedal, Optimization methods for large-scale
  machine learning, SIAM Review 60~(2) (2018) 223--311.

\bibitem{karimireddy2018global}
S.~P. Karimireddy, S.~U. Stich, M.~Jaggi, Global linear convergence of newton's
  method without strong-convexity or lipschitz gradients, arXiv preprint
  arXiv:1806.00413.

\bibitem{marteau2019globally}
U.~Marteau-Ferey, F.~Bach, A.~Rudi, Globally convergent newton methods for
  ill-conditioned generalized self-concordant losses, in: Advances in Neural
  Information Processing Systems, 2019, pp. 7636--7646.

\bibitem{gao2019quasi}
W.~Gao, D.~Goldfarb, Quasi-newton methods: superlinear convergence without line
  searches for self-concordant functions, Optimization Methods and Software
  34~(1) (2019) 194--217.

\bibitem{kovalev2020fast}
D.~Kovalev, R.~M. Gower, P.~Richt{\'a}rik, A.~Rogozin, Fast linear convergence
  of randomized bfgs, arXiv preprint arXiv:2002.11337.

\bibitem{jin2020non}
Q.~Jin, A.~Mokhtari, Non-asymptotic superlinear convergence of standard
  quasi-newton methods, arXiv preprint arXiv:2003.13607.

\bibitem{dennis1974characterization}
J.~E. Dennis, J.~J. Mor{\'e}, A characterization of superlinear convergence and
  its application to quasi-newton methods, Mathematics of computation 28~(126)
  (1974) 549--560.

\bibitem{byrd2016stochastic}
R.~H. Byrd, S.~L. Hansen, J.~Nocedal, Y.~Singer, A stochastic quasi-newton
  method for large-scale optimization, SIAM Journal on Optimization 26~(2)
  (2016) 1008--1031.

\bibitem{gower2016stochastic}
R.~Gower, D.~Goldfarb, P.~Richt{\'a}rik, Stochastic block bfgs: squeezing more
  curvature out of data, in: International Conference on Machine Learning,
  2016, pp. 1869--1878.

\bibitem{johnson2013accelerating}
R.~Johnson, T.~Zhang, Accelerating stochastic gradient descent using predictive
  variance reduction, in: Advances in neural information processing systems,
  2013, pp. 315--323.

\bibitem{moritz2016linearly}
P.~Moritz, R.~Nishihara, M.~Jordan, A linearly-convergent stochastic l-bfgs
  algorithm, in: Artificial Intelligence and Statistics, 2016, pp. 249--258.

\bibitem{zhao2017stochastic}
R.~Zhao, W.~B. Haskell, V.~Y. Tan, Stochastic l-bfgs: Improved convergence
  rates and practical acceleration strategies, arXiv preprint arXiv:1704.00116.

\bibitem{zhou2017stochastic}
C.~Zhou, W.~Gao, D.~Goldfarb, Stochastic adaptive quasi-newton methods for
  minimizing expected values, in: International Conference on Machine Learning,
  2017, pp. 4150--4159.

\bibitem{meng2020fast}
S.~Y. Meng, S.~Vaswani, I.~H. Laradji, M.~Schmidt, S.~Lacoste-Julien, Fast and
  furious convergence: Stochastic second order methods under interpolation, in:
  International Conference on Artificial Intelligence and Statistics, PMLR,
  2020, pp. 1375--1386.

\bibitem{chen2014large}
W.~Chen, Z.~Wang, J.~Zhou, Large-scale l-bfgs using mapreduce, in: Advances in
  Neural Information Processing Systems, 2014, pp. 1332--1340.

\bibitem{berahas2016multi}
A.~S. Berahas, J.~Nocedal, M.~Tak{\'a}c, A multi-batch l-bfgs method for
  machine learning, in: Advances in Neural Information Processing Systems,
  2016, pp. 1055--1063.

\bibitem{bollapragada2018progressive}
R.~Bollapragada, D.~Mudigere, J.~Nocedal, H.-J.~M. Shi, P.~T.~P. Tang, A
  progressive batching l-bfgs method for machine learning, arXiv preprint
  arXiv:1802.05374.

\bibitem{eisen2016decentralized}
M.~Eisen, A.~Mokhtari, A.~Ribeiro, A decentralized quasi-newton method for dual
  formulations of consensus optimization, in: Decision and Control (CDC), 2016
  IEEE 55th Conference on, IEEE, 2016, pp. 1951--1958.

\bibitem{eisen2017decentralized}
M.~Eisen, A.~Mokhtari, A.~Ribeiro, Decentralized quasi-newton methods, IEEE
  Transactions on Signal Processing 65~(10) (2017) 2613--2628.

\bibitem{soori2020dave}
S.~Soori, K.~Mishchenko, A.~Mokhtari, M.~M. Dehnavi, M.~Gurbuzbalaban, Dave-qn:
  A distributed averaged quasi-newton method with local superlinear convergence
  rate, in: International Conference on Artificial Intelligence and Statistics,
  2020, pp. 1965--1976.

\bibitem{najafabadi2017large}
M.~M. Najafabadi, T.~M. Khoshgoftaar, F.~Villanustre, J.~Holt, Large-scale
  distributed l-bfgs, Journal of Big Data 4~(1) (2017) 22.

\bibitem{wright1999numerical}
S.~J. Wright, J.~Nocedal, Numerical optimization, Springer Science 35~(67-68)
  (1999) 7.

\bibitem{lian2015asynchronous}
X.~Lian, Y.~Huang, Y.~Li, J.~Liu, Asynchronous parallel stochastic gradient for
  nonconvex optimization, in: Advances in Neural Information Processing
  Systems, 2015, pp. 2737--2745.

\bibitem{liu2015asynchronous}
J.~Liu, S.~J. Wright, C.~R{\'e}, V.~Bittorf, S.~Sridhar, An asynchronous
  parallel stochastic coordinate descent algorithm, The Journal of Machine
  Learning Research 16~(1) (2015) 285--322.

\bibitem{hsieh2015passcode}
C.-J. Hsieh, H.-F. Yu, I.~Dhillon, Passcode: Parallel asynchronous stochastic
  dual co-ordinate descent, in: International Conference on Machine Learning,
  2015, pp. 2370--2379.

\bibitem{zhao2016fast}
S.-Y. Zhao, W.-J. Li, Fast asynchronous parallel stochastic gradient descent: A
  lock-free approach with convergence guarantee., in: AAAI, 2016, pp.
  2379--2385.

\bibitem{harikandeh2015stopwasting}
R.~Harikandeh, M.~O. Ahmed, A.~Virani, M.~Schmidt, J.~Kone{\v{c}}n{\`y},
  S.~Sallinen, Stopwasting my gradients: Practical svrg, in: Advances in Neural
  Information Processing Systems, 2015, pp. 2251--2259.

\bibitem{xiao2014proximal}
L.~Xiao, T.~Zhang, A proximal stochastic gradient method with progressive
  variance reduction, SIAM Journal on Optimization 24~(4) (2014) 2057--2075.

\end{thebibliography}

\end{document}